\documentclass[sn-mathphys-num]{sn-jnl}

\usepackage{graphicx}%
\usepackage{multirow}%
\usepackage{amsmath,amssymb,amsfonts}%
\usepackage{amsthm}%
\usepackage{mathrsfs}%
\usepackage[title]{appendix}%
\usepackage{xcolor}%
\usepackage{textcomp}%
\usepackage{manyfoot}%
\usepackage{booktabs}%
\usepackage{algorithm}%
\usepackage{algorithmicx}%
\usepackage{algpseudocode}%
\usepackage{listings}%
\usepackage{tikz-cd}
\usepackage{tikz-qtree,tikz-qtree-compat}

\theoremstyle{thmstyleone}%
\newtheorem{theorem}{Theorem}
\newtheorem{proposition}[theorem]{Proposition}%

\newtheorem{lemma}[theorem]{Lemma}

\theoremstyle{thmstyletwo}%

\theoremstyle{thmstylethree}%
\begin{document}

\title[Article Title]{Equivariant gluing theory on regular instanton moduli spaces}


\author{\fnm{Shuaige} \sur{Qiao}}



\abstract{We follow the idea of gluing theory in instanton moduli spaces and discuss the case when there is a finite group $\Gamma$ acting on the 4-manifolds $X_1, X_2$ with $x_1, x_2$ as isolated fixed points, how to glue two $\Gamma$-invariant ASD connections over $X_1, X_2$ together to get a $\Gamma$-invariant ASD connection on the connected sum $X_1\# X_2$.}

\keywords{instanton moduli space, equivariant gluing theory}


\pacs[MSC Classification]{57R18, 57R57, 81T13}

\maketitle

\section{Introduction}
From the late 1980s, gauge theoretic techniques were applied in the area of finite group actions on 4-manifolds. \cite{FURUTA198935} showed that on $S^4$, there is no smooth finite group action with exactly 1 fixed point by arguing that instanton-one invariant connections form a 1-manifold whose boundary can be identified as fixed points of the group action. In \cite{Buchdahl1990}, gauge theoretic techniques were used in studying fixed points of a finite group action on 3-manifolds. \cite{e6fd0a6e-64c3-3748-882e-9b5c2b5bf060} studied pseudofree orbifolds using ASD moduli spaces. A 4-dimensional pseudofree orbifold is a special kind of orbifold which can be expressed as $M^5/S^1$, a quotient of a pseudofree $S^1$-action on a 5-manifold $M^5$. In \cite{A90} Austin studied the orbifold $S^4/\mathbb Z_\alpha$, which is a compactification of $L(\alpha,\beta)\times\mathbb R$ where $L(\alpha,\beta)$ is a Lens space. He also gave a criterion for the existence of instantons on $S^4/\mathbb Z_\alpha$ and calculated the dimension of the instanton moduli space. A more general kind of orbifold, orbifold with isolated fixed points, was discussed in \cite{Furuta1992}, especially when the group-action around each singular point is a cyclic group.

In the study of instanton moduli spaces, gluing theory tells us that given two anti-self-dual connections $A_1, A_2$ on 4-manifolds $X_1, X_2$ respectively, we can glue them together to get a new ASD connection on the space $X_1\# X_2$. It plays an important role in the process of compactifying moduli spaces. This paper follows the idea of gluing theory (cf. Chapter 7 of \cite{DK90}) and discusses the case when there is a finite group $\Gamma$ acting on the 4-manifolds $X_1, X_2$ with $x_1, x_2$ as isolated fixed points, how to glue two ASD $\Gamma$-invariant ASD connections over $X_1, X_2$ together to get an ASD $\Gamma$-invariant connection on $X_1\# X_2$. 

The main differences between the original gluing theory and the $\Gamma$-equivariant case are the following. Firstly, over the fixed points $x_1$ and $x_2$, the $\Gamma$-actions induce two isotropy representations, which are required to be equivalent. Secondly, the gluing parameter depends on the isotropy representations. Finally, we need to deal with the regularity of $A_i$ in the $\Gamma$-invariant spaces.

\section{Set-up}
Suppose $X_1,X_2$ are smooth, oriented, compact, Riemannian 4-manifolds, and $P_1, P_2$ are principal $G$-bundles over $X_1,X_2$ respectively where $G=SU(2)$. Let $\Gamma$ be a finite group acting on $P_i, X_i$ from the left which is smooth and orientation preserving and such that the action on $P_i$ cover the action on $X_i$. 
\[
\begin{tikzcd}
P_i\arrow[loop left,"\Gamma"]\arrow[loop right,"G"]\arrow[d]\\
~~~~~~~~~X_i~\ni~x_i
\end{tikzcd}
\]

\noindent $x_1\in X_1^{\Gamma}$, $x_2\in X_2^{\Gamma}$ are two isolated fixed points with equivalent isotropy representations. i.e., there exists $h\in G$ such that 
\begin{equation}\label{h in the equivalent isotropy reprensentation}
\rho_2(\gamma)=h\rho_1(\gamma)h^{-1}~~~\forall \gamma\in\Gamma
\end{equation}
where $\rho_1,\rho_2$ are isotropy representations of $\Gamma$ at $x_1,x_2$ respectively.

Now we fix two metrics $g_1,g_2$ on $X_1,X_2$ such that the $\Gamma$-action preserves the metrics. This can be achieved by the following lemma.
\begin{lemma}
    For any Riemannian metric $g$ on $X$, 
\[
\tilde g:=\frac{1}{|\Gamma|}\sum_{\gamma\in\Gamma}\gamma^*g
\]
defines a $\Gamma$-invariant metric.
\end{lemma}
The proof is straightforward. We omit it here.

\section{Glue bundles and get an approximate ASD connection $A'$}
The first step is to glue manifolds $X_1$ and $X_2$ by connecting sum. 

Fix a large enough constant $T$ and a small enough constant $\delta$. Let $\lambda>0$ be a constant satisfying $\lambda e^{\delta}\le\frac{1}{2}b$ where $b:=\lambda e^T$. We first glue $X'_1:=X_1\setminus B_{x_1}(\lambda e^{-\delta})$ and $X'_2:=X_2\setminus B_{x_2}(\lambda e^{-\delta})$ together as shown in Figure \ref{figure of gluing X1 and X2 together},
  \begin{figure}[h!]
  \centering 
  \includegraphics[width=10cm, height=8cm]{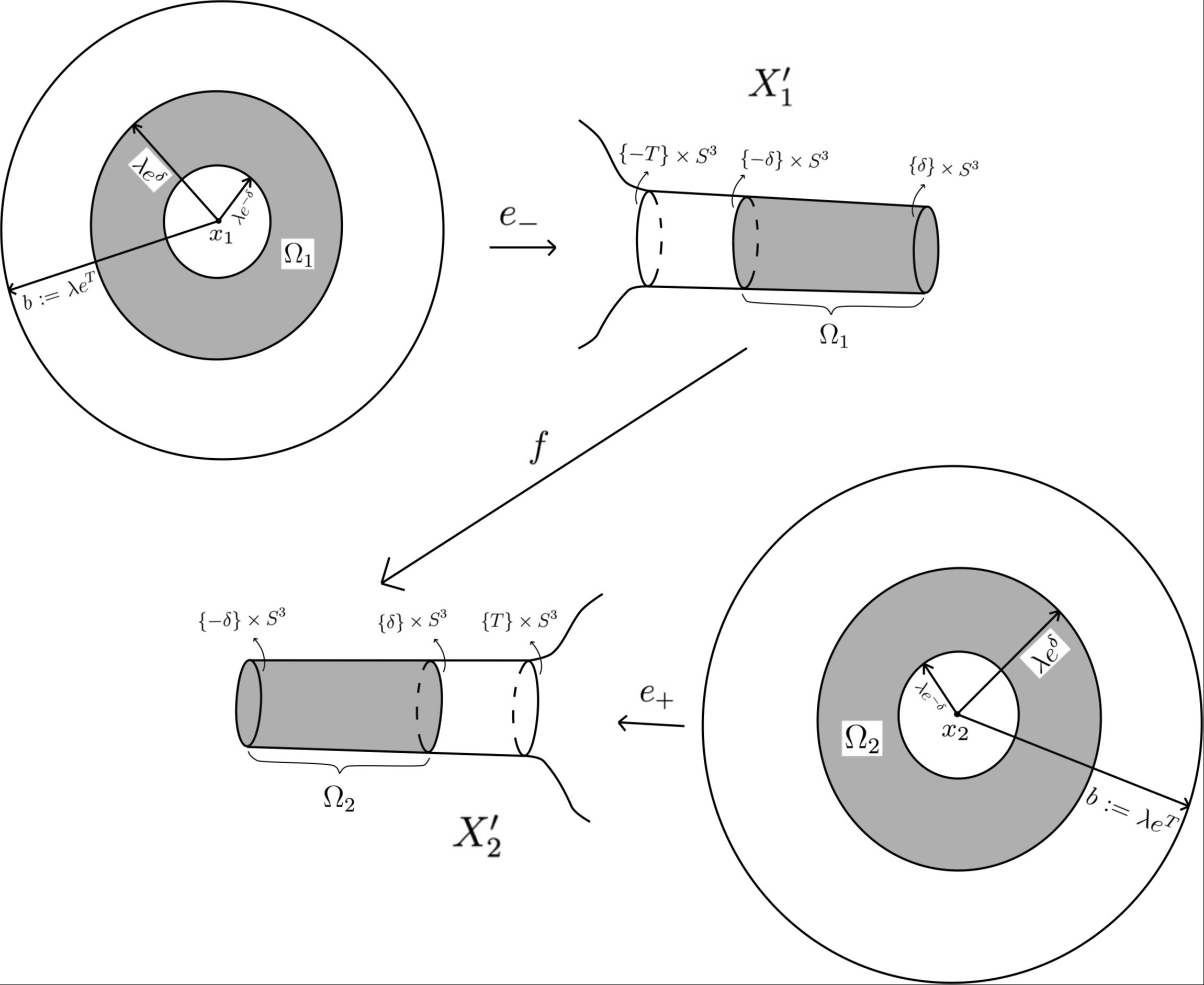}
  \caption{}
  \label{figure of gluing X1 and X2 together}
  \end{figure}
\noindent where $e_\pm$ are defined in polar coordinates by
\begin{eqnarray*}
e_{\pm}:\mathbb R^4\setminus\{0\}&\to& \mathbb R\times S^3 \\
rm&\mapsto&(\pm\log\frac{r}{\lambda},m)
\end{eqnarray*}
and
\begin{equation}\label{identification map in gluing X1X2}
f:\Omega_1=(-\delta,\delta)\times S^3\to\Omega_2=(-\delta,\delta)\times S^3
\end{equation}
is defined to be a $\Gamma$-equivariant conformal map that fixes the first component. Denote the connected sum by $X_1\#_\lambda X_2$ or $X$.

On the new manifold $X$, we define the metric $g_\lambda$ to be a weighted average of $g_1$, $g_2$ on $X_1$ and $X_2$, compared by the diffeomorphism $f$. If $g_\lambda=\sum m_ig_i$ on $X'_i$, we can arrange $1\le m_i\le 2$. This means points are further away from each other on the gluing area.

We now turn to the bundles $P_i$. Suppose $A_i$ are ASD $\Gamma$-invariant connections on $P_i$. We want to glue $P_i|_{X'_i}$ together so that $A_1$ and $A_2$ match on the overlapping part.

The first step is to replace $A_i$ by two $\Gamma$-invariant connections which are flat on the annuli $\Omega_i$. Define the cut-off function $\eta_i$ on $X_i$ such that 
\begin{equation}\label{cut off function eta in gluing1}
\eta_1(x)=\begin{cases}
0~~~x\in [-\delta,+\infty]\times S^3\\
1~~~x\in X_1\setminus B_{x_1}(b)~,
\end{cases}~~
\eta_2(x)=\begin{cases}
0~~~x\in [-\infty,\delta]\times S^3\\
1~~~x\in X_2\setminus B_{x_2}(b)~,
\end{cases}
\end{equation}
and $\eta_i(x)$ depends only on $|x-x_i|$ when $x\in B_{x_i}(b)$. Therefore $\eta_i$ are $\Gamma$-invariant.

Recall from \cite{cmp/1103920742} that in the Euclidean ball $B^4$, an exponential gauge with respect to a connection $A$ is a gauge under which $A(0)=0$ and $A(\partial_r)=0$. Choose an exponential gauge on $B_{x_i}(b)$ and define $A'_i = \eta_i A_i$. Then we have 
\begin{equation*}\label{eq4}
||F^+_{A'_i}||_{L^2} \le \bigg(vol(B_{x_i}(b))\max|F^+_{A'_i}(y)|\bigg)^{\frac{1}{2}}
\le\text{const}\cdot b^2,
\end{equation*}
i.e. $A'_i$ is almost ASD.

The next step is to glue $P_1|_{X'_1}$ and $P_2|_{X'_2}$ together to get a principal $G$-bundle $P$ over $X$ and glue $A'_1$ and $A'_2$ together to get a $\Gamma$-invariant connection $A'$ on $P$.
\begin{lemma}\label{lemma2.2}
There exists a canonical $(\Gamma,G)$-equivariant map $\varphi$:
\begin{eqnarray*}
G\cong P_1|_{x_i}&\xrightarrow{\varphi}&P_2|_{x_2}\cong G\\
g&\mapsto&hg,
\end{eqnarray*}
where $h$ is defined in (\ref{h in the equivalent isotropy reprensentation}) and $(\Gamma,G)$-equivariant means $\varphi$ is $\Gamma$-equivariant and $G$-equivariant.
\end{lemma}
\begin{proof}The $G$-equivariance is obvious and the $\Gamma$-equivariance follows from:
\[
P_1|_{x_1}\xrightarrow{\varphi}P_2|_{x_2}\xrightarrow{\gamma}P_2|_{x_2}~~,~~~~~~~~~P_1|_{x_1}\xrightarrow{\gamma}P_1|_{x_1}\xrightarrow{\varphi}P_2|_{x_2}
\]
\[
~~~~~g~~~\mapsto ~hg~~\mapsto\rho_2(\gamma)hg~~~~~~~~~~g~\mapsto\rho_1(\gamma)g\mapsto ~h\rho_1(\gamma)g
\]
and $\rho_2(\gamma)hg=h\rho_1(\gamma)h^{-1}hg=h\rho_1(\gamma)g$ for any $\gamma\in\Gamma$.
\end{proof}
Denote the subgroup of $(\Gamma,G)$-equivariant gluing parameters by\label{space of gamma g equivariant gluing parameters}
\begin{equation}\label{definition of space of equivariant gluing parameters}
Gl^\Gamma:=Hom_{(\Gamma,G)}(P_1|_{x_1},P_2|_{x_2}).
\end{equation}

\begin{proposition}\label{invariant gluing parameter}
The subgroup of $(\Gamma,G)$-equivariant gluing parameters $Gl^\Gamma$ takes three forms:
\begin{equation}\label{gamma equiv gluing parameter}
Gl^\Gamma\cong
\begin{cases}
G~~~~~~~\text{if }\rho_1(\Gamma),\rho_2(\Gamma)\subset C(G),\\
U(1)~~~\text{if }\rho_1(\Gamma),\rho_2(\Gamma)\not\subset C(G)\text{ and are contained in some }U(1)\subset G,\\
C(G)~~\text{if }\rho_1(\Gamma),\rho_2(\Gamma)\text{ are not contained in any }U(1)\text{ subgroup in }G,
\end{cases}
\end{equation}
where $C(G)$ is the center of $G$.
\end{proposition}
\begin{proof}
By formula (\ref{h in the equivalent isotropy reprensentation}), $\rho_1(\Gamma)$ and $\rho_2(\Gamma)$ are isomorphic and have isomorphic centralisers. For any element $h'$ in the centraliser of $\rho_1(\Gamma)$, $\varphi':g\mapsto hh'g$ is also a $(\Gamma,G)$-equivariant map between $P_1|_{x_1}$ and $P_2|_{x_2}$ since for all $\gamma\in\Gamma$
\[
hh'\rho_1(\gamma)g=h\rho_1(\gamma)h'g=\rho_2(\gamma)hh'g.
\]
For any element $\varphi'\in Gl^\Gamma$, it can be written as $g\mapsto h'g$ for some $h'\in G$. Then $h^{-1}h'$ is in the centraliser of $\rho_1(\Gamma)$ since for any $\gamma\in\Gamma$, $g\in G$, we have
\begin{align*}
\rho_2(\gamma)h'g=h'\rho_1(\gamma)g~~&\Rightarrow~~h^{-1}\rho_2(\gamma)h'g=h^{-1}h'\rho_1(\gamma)g\\
&\Rightarrow~~\rho_1(\gamma)h^{-1}h'g=h^{-1}h'\rho_1(\gamma)g,
\end{align*}
which implies $h^{-1}h'$ commutes with $\rho_1(\gamma)$.

Therefore $Gl^\Gamma$ is isomorphic to the centraliser of $\rho_1(\Gamma)$ in $G$. The three cases in (\ref{gamma equiv gluing parameter}) are the only three groups that are centraliser of some subgroup in $G$ when $G=SU(2)$.
\end{proof}

Recall that annuli $\Omega_i$ are identified by $f:\Omega_1\to\Omega_2$ defined in (\ref{identification map in gluing X1X2}). 
Take $\varphi\in Gl^\Gamma$, we glue $P_1|_{\Omega_1}$ and $P_2|_{\Omega_2}$ together to get $P:=(P_1|_{X'_1})\#_\varphi(P_2|_{X'_2})$ 
with a $\Gamma$-action;
glue $A'_{1}$ and $A'_{2}$ together to get a $\Gamma$-invariant $A'(\varphi)$.
For different gluing parameter $\varphi_1,\varphi_2$, $A'(\varphi_1)$ and $A'(\varphi_2)$ are gauge equivalent if and only if the parameters $\varphi_1$, $\varphi_2$ are in the same orbit of the action of $\Gamma_{A_1}\times\Gamma_{A_2}$ on $Gl$.
We denote $A'(\varphi)$ by $A'$ when the gluing parameter is contextually clear.

\section{Constructing an ASD connection from $A'$}
The general idea is to find a solution $a\in\Omega^1(X,adP)^{\Gamma}$ so that $A:=A'+a$ is anti-self-dual, i.e.,
\begin{equation}\label{eq1}
F_A^+=F^+_{A'}+d^+_{A'}a+(a\wedge a)^+=0.
\end{equation}
To do so, we wish to find a right inverse $R^{\Gamma}$ of $d^+_{A'}$ and an element $\xi\in\Omega^{2,+}(X,adP)^{\Gamma}$ satisfying
\begin{equation}\label{eq2}
F^+_{A'}+\xi+(R^{\Gamma}\xi\wedge R^{\Gamma}\xi)^+=0.
\end{equation}
Then $a=R^{\Gamma}\xi$ is a solution of equation (\ref{eq1}).

Since $A_i$ are two ASD connections, we have the complex:
\[
0\to \Omega^0(X_i,adP_i)\xrightarrow{d_{A_i}} \Omega^1(X_i,adP_i)\xrightarrow{d_{A_i}^+}\Omega^{2,+}(X_i,adP_i)\to 0.
\]
We assume that the second cohomology classes $H^2_{A_1}$, $H^2_{A_2}$ are both zero. The $\Gamma$-action can be induced on this chain complex naturally. It is worth mentioning that the $\Gamma$-action preserves the metric, so the space $\Omega^{2,+}(X_i,adP_i)$ is $\Gamma$-invariant. Define the following two averaging maps:
\begin{eqnarray*}
ave:\Omega^1(X_i,adP_i)&\to&\Omega^1(X_i,adP_i)^{\Gamma}\\
a&\mapsto&\frac{1}{|\Gamma|}\sum_{\gamma\in\Gamma}\gamma^*a\\
ave:\Omega^{2,+}(X_i,adP_i)&\to&\Omega^{2,+}(X_i,adP_i)^{\Gamma}\\
\xi&\mapsto&\frac{1}{|\Gamma|}\sum_{\gamma\in\Gamma}\gamma^*\xi.
\end{eqnarray*}
Note that these maps are surjective since any $\Gamma$-invariant element is mapped to itself.

\begin{proposition}\label{1}The following diagram

\begin{tikzcd}
0\arrow{r}&\Omega^0(X_i,adP_i)\arrow{r}{d_{A_i}}&\Omega^1(X_i,adP_i)\arrow{r}{d_{A_i}^+}\arrow[two heads]{d}{ave}&\Omega^{2,+}(X_i,adP_i)\arrow{r}\arrow[two heads]{d}{ave}&0\\
0\arrow{r}&\Omega^0(X_i,adP_i)^{\Gamma}\arrow{r}{d_{A_i}}&\Omega^1(X_i,adP_i)^{\Gamma}\arrow{r}{d_{A_i}^+}&\Omega^{2,+}(X_i,adP_i)^{\Gamma}\arrow{r}&0
\end{tikzcd}\\
commutes.
\end{proposition}
\begin{proof}
It suffices to show that $d_{A_i}:\Omega^1(X_i,adP_i)\to\Omega^2(X_i,adP_i)$ and $\gamma$ commute for any $\gamma\in\Gamma$. For any $\eta\in\Omega^1(X_i,adP_i)$, we treat $\eta$ as a Lie algebra valued 1-form on $P_i$, then
\begin{eqnarray*}
(d+A_i)(\gamma^*\eta) = \gamma^*d\eta+[A_i,\gamma^*\eta] = \gamma^*d\eta+[\gamma^*A_i,\gamma^*\eta] =  \gamma^*\big((d+A_i)\eta\big).
\end{eqnarray*}
\end{proof}
By Proposition \ref{1},
\[
\text{Im}(d^+_{A_i}\circ ave)=\text{Im}(ave\circ d^+_{A_i})=\Omega^{2,+}(X_i,adP_i)^{\Gamma}.
\]
Therefore, $(H^2_{A_i})^{\Gamma}=0$ and there exists right inverses
\[
R_i^{\Gamma}:\Omega^{2,+}(X_i,adP_i)^{\Gamma}\to\Omega^1(X_i,adP_i)^{\Gamma}
\]
to $d_{A_i}^+$. 

\begin{proposition}\label{3}
$R_i^{\Gamma}$ are bounded operators from $\Omega^{2,+}_{L^2}(X_i,adP_i)^{\Gamma}$ to $\Omega^1_{L^2_1}(X_i,adP_i)^{\Gamma}$.
\end{proposition}
The proof of Proposition \ref{3} follows from Proposition 2.13 of Chapter III of \cite{LawsonMichelsohn+1990} and the fact that the $X_i$ are compact. 

By the Sobolev embedding theorem, we have
\[
||R_i^{\Gamma}\xi||_{L^4}\le\text{const.}||R_i^{\Gamma}\xi||_{L_1^2},
\]
and combined with Proposition \ref{3}, we have
\begin{equation}\label{eq5}
||R_i^{\Gamma}\xi||_{L^4}\le\text{const.}||\xi||_{L^2}.
\end{equation}
Define two operators $Q_i^{\Gamma}:\Omega^{2,+}(X_i,adP_i)^{\Gamma}\to\Omega^1(X_i,adP_i)^{\Gamma}$ by
\[
Q_i^{\Gamma}(\xi):=\beta_iR_i^{\Gamma}\gamma_i(\xi),
\]
where $\beta_i,\gamma_i$ are cut-off functions defined in the Figure \ref{cut off functions beta and gamma} where $\beta_1$ varies on $(1,\delta)\times S^3$, $\beta_2$ varies on $(-\delta,-1)\times S^3$ and $\gamma_i$ varies on $(-1,1)\times S^3$. We can choose $\beta_i$ such that $\left|\displaystyle\frac{\partial\beta_i}{\partial t}\right|<\displaystyle\frac{2}{\delta}$ pointwise, then
\begin{equation}\label{eq6}
||\nabla\beta_i||_{L^4}\le4\pi\left(\int_1^{\delta}\frac{2^4}{\delta^4}dt\right)^{1/4}<64\pi\delta^{-3/4}.
\end{equation}
  \begin{figure}[h!]
  \centering 
  \includegraphics[width=15cm, height=6cm]{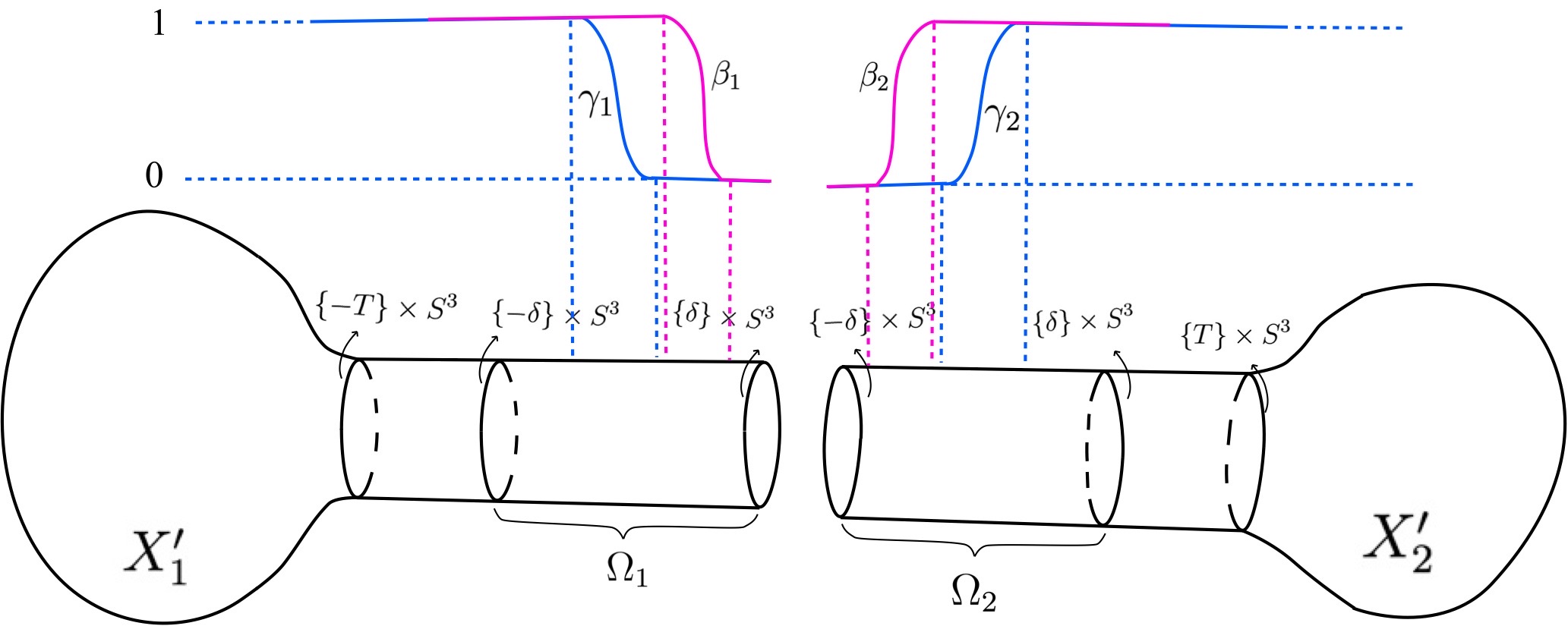}
  \caption{}
  \label{cut off functions beta and gamma}
  \end{figure}
We can choose $\gamma_i$ such that $\gamma_1+\gamma_2=1$ on $\Omega_1\#_f\Omega_2$ where $f$ is defined in (\ref{identification map in gluing X1X2}).

Now we want to extend the operators $Q_i^{\Gamma}$ to $X=X_1\#_\lambda X_2$. Firstly, extend $\beta_i,\gamma_i$ to $X$ in the obvious way. It is worth mentioning that after the extension $\gamma_1+\gamma_2=1$ on $X$. Secondly, for any $\xi\in\Omega^{2,+}(X,adP)^{\Gamma}$, $\gamma_i\xi$ is supported on $X'_i$, thus $R_i^{\Gamma}\gamma_i\xi$ makes sense. Finally, extend $\beta_iR_i^{\Gamma}\gamma_i(\xi)$ to the whole $X$. Therefore $Q_i^{\Gamma}$ can be treated as an operator:
\[
Q_i^{\Gamma}:\Omega^{2,+}(X,adP)^{\Gamma}\to\Omega^1(X,adP)^{\Gamma}.
\]
Define
\[
Q^{\Gamma}:=Q_1^{\Gamma}+Q_2^{\Gamma}:\Omega^{2,+}(X,adP)^{\Gamma}\to\Omega^1(X,adP)^{\Gamma}.
\]
\begin{lemma}\label{lemma3.4}
With definitions above, we have $\forall\xi\in\Omega^{2,+}(X,adP)^{\Gamma}$,
\[
||d_{A'}^+Q^{\Gamma}(\xi)-\xi||_{L^2}\le \text{const.}(b^2+\delta^{-3/4})||\xi||_{L^2}.
\]
\end{lemma}
\begin{proof}
\begin{eqnarray*}
&&||d_{A'}^+Q^{\Gamma}(\xi)-\xi||_{L^2}\\
&=&||d_{A'}^+(Q_1^{\Gamma}(\xi)+Q_2^{\Gamma}(\xi))-\gamma_1\xi-\gamma_2\xi||_{L^2}\\
&=&||d_{A'_1}^+Q_1^{\Gamma}(\xi)+d_{A'_2}^+Q_2^{\Gamma}(\xi)-\gamma_1\xi-\gamma_2\xi||_{L^2}\\
&\le&||d_{A'_1}^+Q_1^{\Gamma}(\xi)-\gamma_1\xi||_{L^2}+||d_{A'_2}^+Q_2^{\Gamma}(\xi)-\gamma_2\xi||_{L^2}.
\end{eqnarray*}
Suppose $A'_i=A_i+a_i$, then
\begin{eqnarray*}
d_{A'_i}^+Q_i^{\Gamma}\xi&=&d_{A_i}^+\beta_iR_i^{\Gamma}\gamma_i\xi+[a_i,\beta_iR_i^{\Gamma}\gamma_i\xi]^+\\
&=&\beta_id_{A_i}^+R_i^{\Gamma}\gamma_i\xi+\nabla\beta_iR_i^{\Gamma}\gamma_i\xi+[\beta_ia_i,R_i^{\Gamma}\gamma_i\xi]^+.
\end{eqnarray*}
The three terms on the right hand side have the following estimates.
\begin{enumerate}
\item[(i).]$\beta_id_{A_i}^+R_i^{\Gamma}\gamma_i\xi=\beta_i\gamma_i\xi=\gamma_i\xi$.
\item[(ii).]$||\nabla\beta_iR_i^{\Gamma}\gamma_i\xi||_{L^2}\le||\nabla\beta_i||_{L^4}||R_i^{\Gamma}\gamma_i\xi||_{L^4}\le \text{const.}\delta^{-3/4}||\xi||_{L^2}$ by the Sobolev multiplication theorem and (\ref{eq5}) and (\ref{eq6}).
\item[(iii).]$||[\beta_ia_i,R_i^{\Gamma}\gamma_i\xi]^+||_{L^2}\le\text{const.}||a_i||_{L^4}||R_i^{\Gamma}\gamma_i\xi||_{L^4}\le\text{const.}b^2||\xi||_{L^2}$ by the Sobolev multiplication theorem and (\ref{eq3}) and (\ref{eq5}).
\end{enumerate}
Therefore $||d_{A'_i}^+Q_i^{\Gamma}(\xi)-\gamma_i\xi||_{L^2}\le\text{const.}(b^2+\delta^{-3/4})||\xi||_{L^2}$ and the result follows.
\end{proof}

The result of Lemma \ref{lemma3.4} means that $Q^{\Gamma}$ is almost a right inverse of $d_{A'}^+$. Next we show there is a right inverse $R^{\Gamma}$ of $d_{A'}^+$.

By Lemma \ref{lemma3.4}, we can choose $\delta$ large enough and $b$ small enough so that $||d_{A'}^+Q^{\Gamma}(\xi)-\xi||_{L^2}\le 2/3||\xi||_{L^2}$, which implies
\[
1/3||\xi||_{L^2}\le||d_{A'}^+Q^{\Gamma}(\xi)||_{L^2}\le 5/3||\xi||_{L^2}.
\]
Then $d_{A'}^+Q^{\Gamma}$ is invertible and
\begin{equation}\label{eq7}
1/3||(d_{A'}^+Q^{\Gamma})^{-1}(\xi)||_{L^2}\le||\xi||_{L^2}.
\end{equation}
Define $R^{\Gamma}:=Q^{\Gamma}(d_{A'}^+Q^{\Gamma})^{-1}$, then it is easy to see that $R^{\Gamma}$ is the right inverse of $d_{A'}^+$. Note that $R^{\Gamma}$ depends on the gluing parameter $\varphi$, so we denote the operator by $R^{\Gamma}_{\varphi}$ when the gluing parameter is not contextually clear. 

$R^{\Gamma}$ has the following good estimate:
\begin{eqnarray}\label{eq8}
||R^{\Gamma}\xi||_{L^4}&=&||(Q_1^{\Gamma}+Q_2^{\Gamma})(d_{A'}^+Q^{\Gamma})^{-1}(\xi)||_{L^4}\nonumber\\
&\le&||Q_1^{\Gamma}(d_{A'}^+Q^{\Gamma})^{-1}(\xi)||_{L^4}+||Q_2^{\Gamma}(d_{A'}^+Q^{\Gamma})^{-1}(\xi)||_{L^4}\nonumber\\
&\le&||R_1^{\Gamma}\gamma_1(d_{A'}^+Q^{\Gamma})^{-1}(\xi)||_{L^4}+||R_2^{\Gamma}\gamma_2(d_{A'}^+Q^{\Gamma})^{-1}(\xi)||_{L^4}\nonumber\\
\text{(by (\ref{eq5}))}&\le&\text{const.}||\gamma_1(d_{A'}^+Q^{\Gamma})^{-1}(\xi)||_{L^2}+\text{const.}||\gamma_2(d_{A'}^+Q^{\Gamma})^{-1}(\xi)||_{L^2}\nonumber\\
&\le&\text{const.}||(d_{A'}^+Q^{\Gamma})^{-1}(\xi)||_{L^2}\nonumber\\
\text{(by (\ref{eq7}))}&\le&\text{const.}||\xi||_{L^2}.
\end{eqnarray}
Then we have
\begin{eqnarray}\label{eq9}
&&||(R^{\Gamma}\xi_1\wedge R^{\Gamma}\xi_1)^+-(R^{\Gamma}\xi_2\wedge R^{\Gamma}\xi_2)^+||_{L^2}\nonumber\\
&\le&||R^{\Gamma}\xi_1\wedge R^{\Gamma}\xi_1-R^{\Gamma}\xi_2\wedge R^{\Gamma}\xi_2||_{L^2}\nonumber\\
&=&\frac{1}{2}||(R^{\Gamma}\xi_1+R^{\Gamma}\xi_2)\wedge(R^{\Gamma}\xi_1-R^{\Gamma}\xi_2)+(R^{\Gamma}\xi_1-R^{\Gamma}\xi_2)\wedge(R^{\Gamma}\xi_1+R^{\Gamma}\xi_2)||_{L^2}\nonumber\\
&\le&\text{const.} ||R^{\Gamma}\xi_1-R^{\Gamma}\xi_2||_{L^4}||R^{\Gamma}\xi_1+R^{\Gamma}\xi_2||_{L^4}\nonumber\\
\text{(by (\ref{eq8}))}&\le&\text{const.}||\xi_1-\xi_2||_{L^2}(||\xi_1||_{L^2}+||\xi_2||_{L^2}).
\end{eqnarray}

Define an operator $T:\xi\mapsto -F^+(A')-(R^{\Gamma}\xi\wedge R^{\Gamma}\xi)^+$, then solving equation (\ref{eq2}) means finding a fixed point of the operator $T$. Here we apply the contraction mapping theorem to $T$ to show there exists a unique fixed point of $T$. There are two things to check:
\begin{enumerate}
\item[1.] There is an $r>0$ such that for small enough $b$, $T$ is a map from the ball $B(r)\subset\Omega^{2,+}_{L^2}(X,adP)$ to itself. This follows from
\begin{eqnarray*}
||\xi||_{L^2}<r~~\Rightarrow~~||T\xi||_{L^2}&\le&||F^+(A')||_{L^2}+||R^{\Gamma}\xi\wedge R^{\Gamma}\xi||_{L^2}\\
&\le&const. b^2+||R^\Gamma\xi||^2_{L^4}\\
&\le&const. b^2+const. ||\xi||^2_{L^2}\\
&\le&const. (b^2+r^2)\\
&<&r~~(for~small~b,r~with~b<<r).
\end{eqnarray*}
\item[2.] $T$ is a contraction for sufficiently small $r$, i.e., there exists $\lambda<1$ such that
\[
||T\xi_1-T\xi_2||\le\lambda||\xi_1-\xi_2||~~\forall\xi_1,\xi_2.
\]
This follows from (\ref{eq9}).
\end{enumerate}

Now we have proved that there exists a unique solution to equation (\ref{eq2}).
\begin{theorem}
Suppose $A_1$, $A_2$ are $\Gamma$-invariant ASD connections on $X_1$, $X_2$ respectively with $H^2_{A_i}=0$. Let $\lambda,T,\delta$ be positive real numbers such that $b:=\lambda e^T>2\lambda e^{\delta}$. Then we can make $\delta$ large enough and $b$ small enough so that for any $(\Gamma,G)$-equivariant gluing parameter $\varphi\in Hom_{(G,\Gamma)}(P_1|_{x_1},P_2|_{x_2})$, there exists $a_{\varphi}\in\Omega^1(X,adP)^{\Gamma}$ with $||a_{\varphi}||_{L^4}\le\text{const.}b^2$ such that $A'(\varphi)+a_{\varphi}$ is a $\Gamma$-invariant ASD connection on $X$. Moreover, if $\varphi_1,\varphi_2$ are in the same orbit of $\Gamma_{A_1}\times \Gamma_{A_2}$ on $Gl$, then $A'(\varphi_1)+a_{\varphi_1}, A'(\varphi_2)+a_{\varphi_2}$ are gauge equivalent.
\end{theorem}
\begin{proof}
We only need to prove the last statement.

If $\varphi_1,\varphi_2$ are in the same orbit of $\Gamma_{A_1}\times \Gamma_{A_2}$ on $Gl$, then  $A'(\varphi_1), A'(\varphi_2)$ are gauge equivalent. For some gauge transformation $\sigma$ we have $\sigma^*A'(\varphi_1)=A'(\varphi_2)$. Applying $\sigma^*$ on both sides of the following formula
\[
F^+_{A'(\varphi_1)}+\xi(\varphi_1)+(R^{\Gamma}_{\varphi_1}\xi(\varphi_1)\wedge R^{\Gamma}_{\varphi_1}\xi(\varphi_1))^+=0
\]
gives
\begin{equation}\label{sigma pull back the formula for phi1}
\sigma^*F^+_{A'(\varphi_1)}+\sigma^*\xi(\varphi_1)+\sigma^*(R^{\Gamma}_{\varphi_1}\xi(\varphi_1)\wedge R^{\Gamma}_{\varphi_1}\xi(\varphi_1))^+=0.
\end{equation}
Since $\sigma^*$ and $d_{A'}^+$ commute and $R^{\Gamma}=Q^{\Gamma}(d_{A'}^+Q^{\Gamma})^{-1}$, then $\sigma^*$ and $R^{\Gamma}$ commute. Then (\ref{sigma pull back the formula for phi1}) becomes
\[
F^+_{A'(\varphi_1)}+\sigma^*\xi(\varphi_1)+(R^{\Gamma}_{\varphi_2}\sigma^*\xi(\varphi_1)\wedge R^{\Gamma}_{\varphi_2}\sigma^*\xi(\varphi_1))^+=0
\]
This means $\sigma^*\xi(\varphi_1)$, $\xi(\varphi_2)$ are solutions to $F^+_{A'(\varphi_2)}+\xi+(R^{\Gamma}_{\varphi_2}\xi\wedge  R^{\Gamma}_{\varphi_2}\xi)^+=0$, which implies $\sigma^*\xi(\varphi_1)$=$\xi(\varphi_2)$. The following deduction completes the proof.
\[
\sigma^*\xi(\varphi_1)=\xi(\varphi_2)~~\Rightarrow~~\sigma^*R^{\Gamma}_{\varphi_1}\xi(\varphi_1)=R^{\Gamma}_{\varphi_2}\sigma^*\xi(\varphi_1)=R^{\Gamma}_{\varphi_2}\xi(\varphi_2)
\]
\[
\Rightarrow~~\sigma^*a_{\varphi_1}=a_{\varphi_2}~~\Rightarrow~~\sigma^*(A'(\varphi_1)+a_{\varphi_1})=A'(\varphi_2)+a_{\varphi_2}.
\]
\end{proof}



\section*{Declarations}

The author declares that there is no conflict of interest.

\bibliography{sn-article}

\end{document}